\documentclass[12pt]{article}
\usepackage{adjustbox}
\usepackage{amsmath} 
\usepackage{amsthm} 
\usepackage{amssymb} 

\usepackage[cal=boondoxo,scr=euler]{mathalfa} 
\usepackage[dvipsnames]{xcolor} 
\usepackage{mathtools} 
\usepackage{mma}
\usepackage{tikz,tikz-cd}
\usepackage[backref=page,linktocpage]{hyperref} 
\usepackage{graphicx} 

\usepackage{cleveref}
\usepackage{enumerate} 
\usepackage{float}
\usepackage{setspace}
\usepackage[utf8]{inputenc}
\usepackage[T1]{fontenc}
\usepackage{caption}
\usetikzlibrary{arrows}
\usetikzlibrary{calc,positioning,arrows,decorations.pathreplacing}
\usepackage[margin=1.20in]{geometry} 

\hypersetup{ 
	colorlinks = true,
	linkbordercolor = {white},
	linkcolor = {BrickRed},
	anchorcolor = {black},
	citecolor = {BrickRed},
	filecolor = {cyan},
	menucolor = {BrickRed},
	runcolor = {cyan},
	urlcolor = {black}
}

\newtheoremstyle{teoremas} 
{11pt}
{11pt}
{\itshape}
{}
{\bfseries}
{}
{.5em}
{}

\theoremstyle{teoremas} 
\newtheorem{theorem}{Theorem}[section] 
\newtheorem{corollary}[theorem]{Corollary} 
\newtheorem{lemma}[theorem]{Lemma} 
\newtheorem{proposition}[theorem]{Proposition} 

\newtheoremstyle{definition} 
{11pt}
{11pt}
{}
{}
{\bfseries}
{}
{.5em}
{}

\theoremstyle{definition} 

\crefname{theorem}{theorem}{theorems} 
\Crefname{theorem}{Theorem}{Theorems} 
\crefname{lemma}{lemma}{lemmas} 
\Crefname{lemma}{Lemma}{Lemmas} 
\crefname{proposition}{proposition}{propositions} 
\Crefname{proposition}{Proposition}{Propositions} 

\DeclareMathOperator{\rk}{rk}

\renewcommand{\L}{\mathcal{L}}
\newcommand{\Ind}{\mathrm{Ind}}
\newcommand{\x}{\mathbf{x}}
\newcommand{\ch}{\mathrm{ch}}
\renewcommand{\dim}{\operatorname{dim}}

\newcommand{\M}{\mathrm{M}}

\makeatletter
\def\dual#1{\expandafter\dual@aux#1\@nil}
\def\dual@aux#1/#2\@nil{\begin{tabular}{@{}c@{}}#1\\#2\end{tabular}}
\makeatother

\begin{document}

\begin{center}
{\large \bf Equivariant inverse Kazhdan--Lusztig polynomials of thagomizer matroids}
\end{center}
\begin{center}
 Alice L.L. Gao$^{1}$, Yun Li$^{2}$ and Matthew H.Y. Xie$^{3}$\\[6pt]

 $^{1,2}$School of Mathematics and Statistics,\\
 Northwestern Polytechnical University, Xi'an, Shaanxi 710072, P.R. China

$^{3}$College of Science, \\
 Tianjin University of Technology, Tianjin 300384, P. R. China\\[6pt]

 Email: $^{1}${\tt llgao@nwpu.edu.cn},
 $^{2}${\tt liyun091402@163.com},
 $^{3}${\tt xie@email.tjut.edu.cn}
\end{center}

\noindent\textbf{Abstract.}
In this paper, we focus on the equivariant inverse Kazhdan--Lusztig polynomials of thagomizer matroids, 
a natural family of graphic matroids associated with the complete tripartite graphs $K_{1,1,n}$. 
These polynomials were introduced by Proudfoot as an extension of the Kazhdan--Lusztig theory for matroids.
We derive closed-form expressions for the $\mathfrak{S}_n$-equivariant inverse Kazhdan--Lusztig polynomials of thagomizer matroids 
and present them explicitly in terms of the irreducible representations of $\mathfrak{S}_n$.
As an application, we also provide explicit formulas for the non-equivariant inverse Kazhdan--Lusztig polynomials,
originally defined by Gao and Xie, and give an alternative proof using generating functions. 
Furthermore, we prove that the inverse Kazhdan--Lusztig polynomials 
of thagomizer matroids are log-concave.

\noindent \emph{AMS Classification 2020:} 
05B35, 05E05, 20C30

\noindent \emph{Keywords:} 
Equivariant inverse Kazhdan--Lusztig polynomial, thagomizer matroid, log-concavity, generating function, Frobenius characteristic map.

\noindent \emph{Corresponding Author: Matthew H.Y. Xie}

\section{Introduction}

The study of Kazhdan--Lusztig polynomials for matroids, 
initiated by Elias, Proudfoot, and Wakefield 
\cite{elias2016kazhdan}, has attracted considerable attention in recent years. 
Based on the Kazhdan--Lusztig--Stanley theory, 
Gao and Xie \cite{gao2021inverse} introduced the inverse 
Kazhdan--Lusztig polynomial $Q_{\M}(t)$ for any matroid $\M$. To study the properties of matroid Kazhdan--Lusztig polynomials, Gedeon, Proudfoot, and Young~\cite{gedeon2017equivariant} introduced the notion of equivariant Kazhdan--Lusztig polynomials.
Within this framework, Proudfoot \cite{proudfoot2021equivariant} developed an equivariant Kazhdan--Lusztig--Stanley theory and defined the equivariant inverse Kazhdan--Lusztig polynomial $Q^W_{\M}(t)$ for a $W$-equivariant matroid.

Braden, Huh, Matherne, Proudfoot, and Wang~\cite{braden2020singular} 
proved that the coefficients of $Q^{W}_{\M}(t)$ 
are isomorphism classes of honest representations of $W$. 
Gao, Xie, and Yang~\cite{gao2022equivariant} determined the equivariant inverse Kazhdan--Lusztig polynomials in the case of uniform matroids.
More recently, Karn, Nasr, Proudfoot, and Vecchi~\cite{karn2023equivariant} extended these computations to the setting of paving matroids. Despite these advances, explicit computations of equivariant inverse Kazhdan--Lusztig polynomials remain challenging, 
especially for graphic matroids. 
In this paper, we compute both the equivariant inverse and the (ordinary) inverse Kazhdan--Lusztig polynomials 
for the thagomizer matroids, a family of graphic matroids 
closely related to complete tripartite graphs. 
We further prove that the inverse Kazhdan--Lusztig polynomials 
of thagomizer matroids are log-concave.

Let us first recall some background. Throughout this paper, 
we adopt the notation of Ferroni, Matherne, Stevens, and Vecchi~\cite{ferroni2024hilbert}. 
Given a matroid $\M = (E, \mathscr{B})$, where $E$ is the ground set 
and $\mathscr{B}$ is the set of all bases of $\M$, 
the rank of any subset $A \subseteq E$ is defined by 
$\rk_\M(A) := \max_{B \in \mathscr{B}} |A \cap B|$, 
and the rank of the matroid is given by $\rk(\M) := \rk_\M(E)$. 
A flat of $\M$ is either the entire ground set $E$, or a proper subset 
$F \subsetneq E$ such that adding any element from $E \setminus F$ 
strictly increases its rank. The set of all flats, 
ordered by inclusion, forms a lattice denoted by $\L(\M)$. 
For any flat $F \in \L(\M)$, we denote by $\M|_F$ the restriction 
of $\M$ at $F$, and by $\M/F$ the contraction of $\M$ at $F$. 
Let $\chi_{\M}(t)$ be the characteristic polynomial of $\M$. 
Gao and Xie~\cite{gao2021inverse} showed that for every loopless matroid $\M$ there is a unique polynomial $Q_\M(t)\in\mathbb{Z}[t]$ satisfying
\begin{itemize}
 \item If $\rk(\M)=0$, then $Q_\M(t)=1$.
 \item If $\rk(\M)>0$, then $\deg Q_\M(t)<\tfrac12\rk(\M)$.
 \item For every matroid $\M$,
 $$
 t^{\rk(\M)}(-1)^{\rk(\M)}\,Q_\M\!\bigl(t^{-1}\bigr)
 =
 \sum_{F\in\L(\M)}
 (-1)^{\rk(\M|_F)}\,
 Q_{\M|_F}(t)\,
 t^{\rk(\M/F)}\,
 \chi_{\M/F}\!\bigl(t^{-1}\bigr).
 $$
\end{itemize}
The polynomial $Q_\M(t)$ is called the inverse Kazhdan--Lusztig polynomial of $\M$.

We proceed to recall some definitions and notations related to equivariant inverse Kazhdan--Lusztig polynomials of matroids. 
Let $\M$ be a matroid on a finite ground set $E$, 
and let $W$ be a finite group acting on $E$ and preserving $\M$. 
Gedeon, Proudfoot, and Young~\cite{gedeon2017equivariant} 
referred to this data as an equivariant matroid $W \curvearrowright \M$.
Given an equivariant matroid $W \curvearrowright \M$, Gedeon, Proudfoot, and Young~\cite{gedeon2017equivariant} defined the equivariant characteristic polynomial 
$H^{W}_{\M}(t)\in \mathrm{VRep}(W)[t]$
and 
the equivariant Kazhdan--Lusztig polynomial 
$
P^{W}_{\M}(t) \in \mathrm{VRep}(W)[t]
$.
Here $\mathrm{VRep}(W)$ denotes the ring of virtual representations 
of $W$, and $\mathrm{VRep}(W)[t]$ is the polynomial ring 
in $t$ with coefficients in $\mathrm{VRep}(W)$. 
For each flat $F \in \mathcal{L}(\M)$, let $W_F \subseteq W$ be the stabilizer subgroup of $F$.
Proudfoot~\cite{proudfoot2021equivariant} showed that there exists a unique way to assign to each equivariant matroid $W \curvearrowright \M$ a polynomial $Q^{W}_{\M}(t) \in \mathrm{VRep}(W)[t]$, called the equivariant inverse Kazhdan--Lusztig polynomial, satisfying 
\begin{itemize}
\item If $\rk(\M)=0$, then ${Q^{W}_{\M}}(t)$ is the trivial representation in degree $0$.

\item If $\rk(\M)>0$, then $\deg {Q^{W}_{\M}}(t) < \frac 12 \rk(\M)$.

\item For every $W \curvearrowright \M$, 
\begin{align*}
& t^{\rk(\M)} \cdot (-1)^{\rk(\M)}{Q^{W}_{\M}}(t^{-1}) \\
&\qquad \qquad =\displaystyle\sum_{[F] \in \L(\M)/W}\Ind^{W}_{W_F}\left((-1)^{\rk(\M|_F)} {Q^{W_F}_{\M|_F}}(t) \otimes t^{\rk(\M/F)} {H^{W_F}_{\M/F}}(t^{-1})\right).
\end{align*}
\end{itemize}
Moreover, the graded dimension of $Q^{W}_{\M}(t)$ recovers the usual inverse Kazhdan--Lusztig polynomial $Q_{\M}(t) \in \mathbb{Z}[t]$.

Let $T_n$ denote the matroid associated with 
the complete tripartite graphs $K_{1,1,n}$, which is equivalently obtained 
by adding an edge between the two distinguished vertices 
of the bipartite graph $K_{2,n}$. 
Gedeon~\cite{gedeon2017kazhdan} referred to $T_n$ 
as the thagomizer matroid, and computed its 
Kazhdan--Lusztig polynomial $P_{T_n}(t)$. 
We now focus on the case where $W = \mathfrak{S}_n$ and $\M = T_n$.
The equivariant Kazhdan--Lusztig polynomial 
$P^{\mathfrak{S}_n}_{T_n}(t)$ was determined by Xie and Zhang~\cite{xie2019equivariant}. 
In this paper, we study the equivariant inverse Kazhdan--Lusztig polynomial $Q^{\mathfrak{S}_n}_{T_n}(t)$ of the thagomizer matroid $T_n$. 
Given a partition $\lambda \vdash n$, let $V_{\lambda}$ denote the irreducible representation of the symmetric group $\mathfrak{S}_n$ corresponding to $\lambda$.
Our first main result provides a closed formula for the corresponding 
equivariant inverse Kazhdan--Lusztig polynomial $Q^{\mathfrak{S}_n}_{T_n}(t)$.

\begin{theorem}\label{thm-equi}
For any equivariant thagomizer matroid $\mathfrak S_n \curvearrowright T_n$ with $n \geq 0$, we have
$$ {Q^{\mathfrak{S}_n}_{T_n}}(t) = \sum_{k=0}^{\lfloor \frac{n}{2} \rfloor }\sum_{i=0}^{k} \sum_{j=k-i}^{\lfloor \frac{n-3i}{2} \rfloor} (n-3i-2j+1) V_{(3^i,2^j,1^{n-3i-2j})} t^{k}.$$
\end{theorem}

In the non-equivariant setting, Gao and Xie conjectured that the coefficients of any inverse matroid Kazhdan--Lusztig polynomials are nonnegative. This conjecture was resolved by Braden, Huh, Matherne, 
Proudfoot, and Wang~\cite{braden2020singular}.
Ardila and Sanchez~\cite{ardila2023valuations} 
showed that the inverse Kazhdan--Lusztig polynomial is a 
valuative invariant. Based on this result, 
Ferroni and Schr\"{o}ter~\cite{ferroni2024valuative} derived 
an explicit formula for $Q_{\M}(t)$ in the case where $\M$ is an elementary 
split matroid --- a broad class that includes paving matroids and thus 
uniform matroids.
Our next main result is the following theorem.

\begin{theorem}\label{thm-ord}
For any thagomizer matroid $T_{n}$ with $n \geq 0$, we have
\begin{align}\label{QTN}
 Q_{T_{n}}(t)=
 \sum_{k=0}^{\lfloor\frac{n}{2}\rfloor} \sum_{i=2k}^{n} \frac{n-i+1}{n+1} \dbinom{n+1}{k,i-2k,n+k-i+1} t^{k}.
\end{align}
\end{theorem}

The third part of this paper investigates the log-concavity 
of the inverse Kazhdan--Lusztig polynomials of thagomizer matroids. 
Recall that a polynomial $f(t)=\sum_{i=0}^n a_i t^i$ is log-concave
if its coefficients satisfy $a_i^2 \geq a_{i-1}a_{i+1}$ for all $1 \leq i \leq n-1$,
and it is said to have no internal zeros if there are no three indices
$0 \leq i < j < k \leq n$ such that $a_i, a_k \neq 0$ and $a_j = 0$.
Motivated by the log-concavity conjecture for matroid Kazhdan--Lusztig polynomials 
proposed by Elias, Proudfoot, and Wakefield~\cite{elias2016kazhdan}, 
Gao and Xie~\cite{gao2021inverse} 
conjectured that the coefficients of the inverse Kazhdan--Lusztig polynomial 
$Q_\M(t)$ are also log-concave and contain no internal zeros for any matroid $\M$.
Partial progress toward this conjecture has already been achieved. 
Gao and Xie~\cite{gao2021inverse} verified the log-concavity of \(Q_\M(t)\) for uniform matroids, 
whereas Xie and Zhang~\cite{xie2025log} established it for paving matroids.
In this paper, we further these developments by proving the following result.

\begin{theorem}\label{thm-log}
For any positive integer $n$, the inverse Kazhdan--Lusztig polynomial ${Q_{T_n}(t)}$ of the thagomizer matroid $T_n$ is log-concave.
\end{theorem}

This paper is organized as follows.
In Section~\ref{section:2}, we will review some basic definitions
and notation for symmetric functions and give some results which will be used later.
In Section~\ref{sec-equ-inv}, 
we compute the equivariant inverse Kazhdan--Lusztig polynomials of thagomizer matroids using
the theory of symmetric functions, from which we also obtain the non-equivariant inverse Kazhdan--Lusztig polynomials.
In Section~\ref{sec-inv}, as an alternative approach,
we determine the inverse Kazhdan--Lusztig polynomials of thagomizer matroids 
by relating them to the Kazhdan--Lusztig polynomials and applying the technique 
of ordinary generating functions. 
In Section~\ref{sec-log}, 
we first derive several recurrence relations satisfied by $d_{n,k}$ --- where $d_{n,k}$, defined in~\eqref{eq:d_n,k}, denotes a component of the coefficient of $t^k$ in the inverse Kazhdan--Lusztig polynomial of thagomizer matroids. We then establish a lower bound for $\frac{d_{n,k}}{d_{n-1,k}}$, 
and ultimately verify the log-concavity of the inverse Kazhdan--Lusztig polynomials of thagomizer matroids.

\section{Symmetric functions}\label{section:2}

In this section, we review the basic definitions and fundamental 
results concerning symmetric functions. For terminology not 
explicitly defined here, we refer the reader to 
Stanley~\cite{Stanley1999} and Haglund~\cite{haglund2008q}. 
We also present several symmetric function identities that will 
be instrumental in evaluating the equivariant inverse 
Kazhdan--Lusztig polynomials of thagomizer matroids.

\subsection{Frobenius characteristic map}

The Frobenius characteristic map is an important tool in the 
theory of symmetric functions. 
Let $\Lambda_n$ denote the $\mathbb{Z}$-module of homogeneous symmetric functions of degree $n$ in variables $\x=(x_1,x_2,\ldots)$. 
The Frobenius characteristic map is an isomorphism
$$\ch:\mathrm{VRep}(\mathfrak{S}_n) \to \Lambda_n.$$ 
This map sends each irreducible representation $V_\lambda$ of the symmetric group $\mathfrak{S}_n$ to the Schur function ${s_\lambda}(\x)$ indexed by the partition $\lambda \vdash n$. 
In particular, the trivial representation $V_{(n)}$ is mapped to 
$s_{(n)}(\x) = h_n(\x)$, the complete symmetric function, while 
the sign representation $V_{(1^n)}$ is mapped to $s_{(1^n)}(\x) = 
e_n(\x)$, the elementary symmetric function.
The Frobenius characteristic map $\ch$ extends 
linearly from $\mathrm{VRep}(\mathfrak{S}_n)$ to the graded representation ring 
$\mathrm{VRep}(\mathfrak{S}_n)[t]$, preserving the grading in $t$.
Moreover, for any two graded virtual representations $V_1 \in \mathrm{VRep}(\mathfrak{S}_{n_1})[t]$ and $V_2 \in \mathrm{VRep}(\mathfrak{S}_{n_2})[t]$, we have
$$\ch\ \mathrm{Ind}_{\mathfrak{S}_{n_1} \times \mathfrak{S}_{n_2}}^{\mathfrak{S}_{n_1+n_2}}(V_1 \otimes V_2) = \ch(V_1) \ch(V_2).$$
This identity extends naturally to tensor products of three 
or more factors.

\subsection{Plethystic substitution}

We now recall the definition of plethystic substitution, 
following the notation of Haglund~\cite{haglund2008q}. Let $E = (t_1, t_2, t_3, \ldots)$ be a formal series of rational functions in the parameters. The plethystic substitution of $E$ into the $k$-th power sum $p_k$ is defined as
$$
p_k[E] = E(t_1^k, t_2^k, \ldots).
$$
For any symmetric function $f$ expressed as $f = \sum_{\lambda} c_{\lambda} p_{\lambda}$, its plethystic substitution is given by
$$
f[E] = \sum_{\lambda} c_{\lambda} \prod_{i} p_{\lambda_i}[E].
$$
In particular, when $X = \sum_i x_i$, we have
$p_k[X] = p_k(\x)$, which implies $f[X] = f(\x)$ for any symmetric function $f$.
It is  straightforward to see that
$p_k[-X] = -p_k(\x) = -p_k[X]$.
\begin{lemma}[\textnormal{\cite[Theorem~1.27]{haglund2008q}}]\label{lem:hung-1}
Let $E= (t_1, t_2, t_3, \ldots)$ and $F= (w_1, w_2, w_3, \ldots)$
be two
formal series of rational functions in the parameters $t_1, t_2, t_3, \ldots$ and $w_1, w_2, w_3, \ldots$. Then 
\begin{align*}
e_{n}[E+F]&=\sum_{j=0}^{n}e_{j}[E]e_{n-j}[F];\\
e_{n}[E-F]&=\sum_{j=0}^{n}(-1)^{n-j}e_{j}[E]h_{n-j}[F].
\end{align*}
\end{lemma}

\subsection{The Pieri rule on Schur functions}

In this subsection, we present two symmetric function identities, 
which are likely known in the literature. For completeness, 
we provide proofs based on the Pieri rule.

We begin by recalling the Pieri rule for Schur functions; 
see Stanley~\cite{Stanley1999} for details. For any $i \geq 1$, 
we have
\begin{align*}
 e_{i}(\x)s_{\lambda}(\x)=\sum_{\mu}s_{\mu}(\x),
\end{align*}
where the sum is over all partitions $\mu$ such that 
$\mu/\lambda$ is a vertical strip of size $i$. Similarly,
\begin{align*}
 h_{i}(\x)s_{\lambda}(\x)=\sum_{\mu}s_{\mu}(\x),
\end{align*}
where the sum is over all partitions $\mu$ such that 
$\mu/\lambda$ is a horizontal strip of size $i$.

We now turn to the first symmetric function identity, which 
can be stated as follows.

\begin{lemma} \label{lemma-pieri-1}
For $m \ge 0$ and $k \ge 1$, we have
$$
 \sum_{j=0}^{m}(-1)^{j}\, e_{m-j}(\x)\, s_{(j+2,2^{k-1})}(\x)
 = s_{(2^{k}, 1^{m})}(\x).
$$
\end{lemma}

\begin{proof}
When $m=0$, the statement reduces to 
$e_0(\x)s_{(2^k)}(\x)=s_{(2^k)}(\x)$, which is immediate. 

Assume $m \ge 1$.
Consider $e_{m-j}(\x) s_{(j+2,2^{k-1})}(\x)$ for $0 \le j \le m$. 
For $j=0$ and $j=m$, the Pieri rule gives
\begin{align}\label{pieri-1}
 e_{m}(\x) s_{(2^{k})}(\x)
 =
 \sum_{x=0}^{\min\{m,k\}} s_{(3^{x}, 2^{k-x}, 1^{m-x})}(\x),
 \qquad
 e_{0}(\x) s_{(m+2,2^{k-1})}(\x)=s_{(m+2,2^{k-1})}(\x).
\end{align}
Now suppose $1 \le j \le m-1$ (hence $m \ge 2$). 
Applying the Pieri rule to a vertical strip yields two contributions, 
depending on whether the top row acquires a box. That is,
\begin{align*}
 s_{(1^{m-j})}(\x) s_{(j+2,2^{k-1})}(\x)
 =&
 \sum_{x=0}^{\min\{k-1,m-j\}}
 s_{(j+2,3^x, 2^{k-1-x}, 1^{m-j-x})}(\x)
\\
&\qquad \qquad
 +
 \sum_{x=0}^{\min\{k-1,m-j-1\}}
 s_{(j+3,3^x, 2^{k-1-x}, 1^{m-j-1-x})}(\x).
\end{align*}
Summing over $j=1,\dots,m-1$ with the coefficient $(-1)^j$ and reindexing the second sum by $j\mapsto j-1$ yields 
\begin{align*}
\sum_{j=1}^{m-1}(-1)^j s_{(1^{m-j})}(\x) s_{(j+2,2^{k-1})}(\x)
&=
\sum_{j=1}^{m-1}\sum_{x=0}^{\min\{k-1,\,m-j\}}(-1)^j s_{(j+2,3^x,2^{k-1-x},1^{m-j-x})}(\x)
\\
&\quad
+\sum_{j=2}^{m}\sum_{x=0}^{\min\{k-1,m-j\}}(-1)^{j-1} s_{(j+2,3^x,2^{k-1-x},1^{m-j-x})}(\x)
\\
&=(-1)^{m-1}s_{(m+2,2^{k-1})}(\x)
 -\sum_{x=1}^{\min\{m,k\}} s_{(3^x,2^{k-x},1^{m-x})}(\x).
\end{align*}
Combining \eqref{pieri-1} with the above expression, we obtain
\begin{align*}
\sum_{j=0}^{m}(-1)^{j} e_{m-j}(\x) s_{(j+2,2^{k-1})}(\x)
=&
(-1)^m s_{(m+2,2^{k-1})}(\x)
+\sum_{x=0}^{\min\{m,k\}} s_{(3^x,2^{k-x},1^{m-x})}(\x)
\\
&
+\Big((-1)^{m-1}s_{(m+2,2^{k-1})}(\x)
 -\sum_{x=1}^{\min\{m,k\}} s_{(3^x,2^{k-x},1^{m-x})}(\x)\Big)
\\
=& s_{(3^0,2^{k-0},1^{m-0})}(\x)
= s_{(2^{k},1^{m})}(\x),
\end{align*}
which completes the proof.
\end{proof}

We now present the second symmetric function identity.

 \begin{lemma}\label{lem-pire-2}
 For $n\geq 0$ and $0 \leq k \leq \lfloor \frac{n}{2} \rfloor $, we have 
 $$\sum_{i=0}^{n-2k} e_{i}(\x) s_{(2^{k}, 1^{n-2k-i})}(\x) = \sum_{i=0}^{k} \sum_{j=k-i}^{\lfloor \frac{n-3i}{2} \rfloor} (n-3i-2j+1) s_{(3^i,2^j,1^{n-3i-2j})}(\x). $$
 \end{lemma}

 \begin{proof}
Fix $n$ and $k$. For $0\le i\le n-2k$, the Pieri rule gives
 \begin{align*}
 e_{i}(\x) s_{(2^{k}, 1^{n-2k-i})}(\x) &= \sum_{x=0}^{\mathrm{min} \{ k, i\}} \sum_{y=0}^{\mathrm{min} \{ i-x, n-2k-i\}}s_{(3^x, 2^{k-x+y}, 1^{n-2k-x-2y})} (\x).
 \end{align*}
Observe that when $x>i$, the inner sum is empty and hence contributes zero. 
 Extending the upper limit of $x$ to $k$ and setting $j = k-x+y$, we obtain
\begin{align*}
 e_{i}(\x) s_{(2^{k}, 1^{n-2k-i})}(\x) &=\sum_{x=0}^{k} \sum_{j=k-x}^{\mathrm{min} \{ i+k-2x, n-k-i-x\}}s_{(3^x, 2^{j}, 1^{n-3x-2j})} (\x).
 \end{align*}
Summing over $i=0,\dots,n-2k$ and exchanging the order of summation yields
 \begin{align*}
 \sum_{i=0}^{n-2k} e_{i}(\x) s_{(2^{k}, 1^{n-2k-i})}(\x) &=
 \sum_{i=0}^{n-2k} \sum_{x=0}^{k} \sum_{j=k-x}^{\mathrm{min} \{ i+k-2x, n-k-i-x\}}s_{(3^x, 2^{j}, 1^{n-3x-2j})}(\x) \\
 &=\sum_{x=0}^{k} \sum_{j=k-x}^{\lfloor \frac{n-3x}{2} \rfloor} \sum_{i=j-k+2x}^{n-k-x-j}s_{(3^x, 2^{j}, 1^{n-3x-2j})}(\x) \\
 &=\sum_{x=0}^{k} \sum_{j=k-x}^{\lfloor \frac{n-3x}{2} \rfloor}(n-3x-2j+1)s_{(3^x, 2^{j}, 1^{n-3x-2j})}(\x).
 \end{align*}
Replacing the index $x$ by $i$ gives the stated identity.
\end{proof}

\section{Equivariant inverse Kazhdan--Lusztig polynomials}\label{sec-equ-inv}

In this section, we prove Theorem~\ref{thm-equi}, which gives an explicit formula for the equivariant inverse Kazhdan-–Lusztig polynomials of the thagomizer matroids under the natural $\mathfrak{S}_n$-action. As a corollary, we also recover the non-equivariant case, stated in Theorem~\ref{thm-ord}.

\subsection{Proof of Theorem \ref{thm-equi}}

We begin by recalling the structure of the flats of $T_n$, originally described by Gedeon~\cite[Section~3]{gedeon2017kazhdan}.
Let $AB$ be the distinguished edge in $T_n$. For each $j\in \{1,\ldots,n\}$, the subgraph consisting of edges
$Aj$ and $Bj$ is called a spike. 
Since $T_n$ has rank $n+1$, any flat $F \in \mathcal{L}(T_n)$ of rank $\rk(F) = i$, where $0 \leq i \leq n+1$, necessarily falls into one of the following two categories:
\begin{itemize}
\item[$(1)$] $F$ contains exactly one edge from each of $i$ distinct spikes;
 or
\item[$(2)$] $F$ consists of the union of $i - 1$ complete spikes together with the distinguished edge $AB$. 
\end{itemize}
In the first case, the restriction $T_n|_F $ is a Boolean matroid of rank $i$, while the contraction $T_n/F$ is a matroid whose lattice of flats is isomorphic to that of $ T_{n-i}$. In the second case, the restriction $T_n|_F$ corresponds to a matroid whose lattice of flats is isomorphic to that of $ T_{i-1}$, and the contraction $T_n/F $ is a Boolean matroid of rank $ n-i+1 $.
Moreover, the number of flats of the first type with rank $i$ is 
$\binom{n}{i} \cdot 2^{i}$, and the number of flats of the second type 
with rank $i+1$ is $\binom{n}{i}$.

Let $P_n(\x;t)$ and $Q_n(\x;t)$ denote the images of the equivariant 
Kazhdan--Lusztig polynomial $P^{\mathfrak{S}_n}_{T_n}(t)$ and the equivariant 
inverse Kazhdan--Lusztig polynomial $Q^{\mathfrak{S}_n}_{T_n}(t)$ under the Frobenius 
characteristic map, respectively. 
An explicit formula for $P^{\mathfrak{S}_n}_{T_n}(t)$ and $P_n(\x;t)$ was obtained by Xie and Zhang~\cite{xie2019equivariant}. 
To prove Theorem~\ref{thm-equi}, we first establish a relation between $Q_n(\x;t)$ and $P_i(\x;t)$.

\begin{lemma}\label{lem-equi}
For any nonnegative integer $n$, we have
\begin{align}\label{lem-equi-equ}
Q_{n}(\x;t)=\sum_{i=0}^{n}(-1)^i e_{n-i}[3X]P_{i}(\x;t).
\end{align}
\end{lemma}

\begin{proof}
Within the framework of equivariant Kazhdan--Lusztig--Stanley theory, Proudfoot~\cite{proudfoot2021equivariant} established that 
\begin{align}\label{equi-equ1}
\sum_{[F] \in \L(\M)/W} (-1)^{\rk(\M|_F) } \Ind^{W}_{W_F}\left({Q^{W_F}_{\M|_F}}(t) \otimes {P^{W_F}_{\M/F}}(t)\right)=0
\end{align}
Applying \eqref{equi-equ1} to the equivariant thagomizer matroid $\mathfrak{S}_n \curvearrowright T_n$, we obtain
\begin{align*}
\sum_{[F] \in \mathcal{L}(T_n)/\mathfrak{S}_n} (-1)^{\rk (T_n|_F) } \Ind^{\mathfrak{S}_n}_{(\mathfrak{S}_n)_F}\left({Q^{(\mathfrak{S}_n)_F}_{T_n|_F}}(t) \otimes {P^{(\mathfrak{S}_n)_F}_{T_n/F}}(t)\right)=0.
\end{align*}
Substituting the classification of flats of $T_n$, we get
\begin{align}
&\sum_{i=0}^{n}(-1)^{i+1}\Ind^{\mathfrak{S}_n}_{\mathfrak{S}_{i} \times \mathfrak{S}_{n-i}}\left({Q^{\mathfrak{S}_{i} \times \mathfrak{S}_{n-i}}_{T_i}}(t) \otimes {P^{\mathfrak{S}_{i} \times \mathfrak{S}_{n-i}}_{B_{n-i}}}(t) \right) \notag \\
&+ \sum_{i+j+m=n}(-1)^{j+m}\Ind^{\mathfrak{S}_n}_{\mathfrak{S}_{j} \times \mathfrak{S}_{m} \times \mathfrak{S}_i} \left(Q^{\mathfrak{S}_{j} \times \mathfrak{S}_{m} \times \mathfrak{S}_i}_{B_j}(t) \otimes Q^{\mathfrak{S}_{j} \times \mathfrak{S}_{m} \times \mathfrak{S}_i}_{B_m}(t) \otimes P^{\mathfrak{S}_{j} \times \mathfrak{S}_{m} \times \mathfrak{S}_i}_{T_i}(t) \right)=0, \label{eq:equiva_thg_1}
\end{align}
where $B_n$ is the Boolean matroid of rank $n$.
Rearranging \eqref{eq:equiva_thg_1}, we have
\begin{align}\label{lem-equi-eq1}
&\sum_{i=0}^{n}(-1)^{i}\Ind^{\mathfrak{S}_n}_{\mathfrak{S}_{i} \times \mathfrak{S}_{n-i}}\left({Q^{\mathfrak{S}_{i}}_{T_i}}(t) \otimes {P^{\mathfrak{S}_{n-i}}_{B_{n-i}}}(t) \right)\nonumber\\
&\qquad \qquad= \sum_{i+j+m=n}(-1)^{m+j}\Ind^{\mathfrak{S}_n}_{\mathfrak{S}_{j} \times \mathfrak{S}_{m} \times \mathfrak{S}_i} \left(Q^{\mathfrak{S}_{j}}_{B_j}(t) \otimes Q^{\mathfrak{S}_{m}}_{B_m}(t) \otimes P^{\mathfrak{S}_i}_{T_i}(t) \right).
\end{align}

Recall that for any equivariant Boolean matroid $\mathfrak{S}_n \curvearrowright B_n$ with $n \geq 0$, we have $P^{\mathfrak{S}_n}_{B_n}(t)=V_{(n)}$ and $Q^{\mathfrak{S}_n}_{B_n}(t)=V_{(1^n)}$; see \cite[Lemma 3.5]{gao2022equivariant} and \cite[Theorem 3.4]{gao2022equivariant}.
Applying the Frobenius characteristic map, the left-hand side of \eqref{lem-equi-eq1} becomes
\begin{align*}
\ch \left(\sum_{i=0}^{n}(-1)^{i}\Ind^{\mathfrak{S}_n}_{\mathfrak{S}_{i} \times \mathfrak{S}_{n-i}}\left({Q^{\mathfrak{S}_{i}}_{T_i}}(t) \otimes {P^{\mathfrak{S}_{n-i}}_{B_{n-i}}}(t) \right)\right)
=&\sum_{i=0}^{n}(-1)^{i}\ch~{Q^{\mathfrak{S}_{i}}_{T_i}}(t) \cdot \ch~{P^{\mathfrak{S}_{n-i}}_{B_{n-i}}}(t)\\
=&\sum_{i=0}^{n}(-1)^{i}Q_{i}(\x;t)h_{n-i}(\x),
\end{align*}
where the last equality follows from $\ch~ V_{(n-i)}=h_{n-i}(\x)$. For the right-hand side of \eqref{lem-equi-eq1}, we obtain
\begin{align*}
&\ch \left(\sum_{i+j+m=n}(-1)^{m+j}\Ind^{\mathfrak{S}_n}_{\mathfrak{S}_{j} \times \mathfrak{S}_{m} \times \mathfrak{S}_i} \left(Q^{\mathfrak{S}_{j}}_{B_j}(t) \otimes Q^{\mathfrak{S}_{m}}_{B_m}(t) \otimes P^{\mathfrak{S}_i}_{T_i}(t) \right) \right)\\
=&\sum_{i+j+m=n}(-1)^{m+j}\ch~Q^{\mathfrak{S}_{j}}_{B_j}(t) \cdot \ch~Q^{\mathfrak{S}_{m}}_{B_m}(t) \cdot \ch~P^{\mathfrak{S}_i}_{T_i}(t)\\
=&\sum_{i+j+m=n}(-1)^{m+j}e_{j}(\x)e_{m}(\x)P_{i}(\x;t),
\end{align*}
where the last equality follows because $\ch~ V_{(1^n)}=e_{n}(\x)$. Thus, \eqref{lem-equi-eq1} is equivalent to
\begin{align}\label{lem-equi-eq2}
\sum_{i=0}^{n}(-1)^{i}Q_{i}(\x;t)h_{n-i}(\x)=\sum_{i+j+m=n}(-1)^{m+j}e_{j}(\x)e_{m}(\x)P_{i}(\x;t).
\end{align}

Now, multiply both sides of \eqref{lem-equi-eq2} by $u^n$ and then sum over all $n \geq 0$. For the left-hand side of \eqref{lem-equi-eq2}, we have
\begin{align}\label{eq:gener-equi-fian-1-1}
\sum_{n=0}^{\infty}\left(\sum_{i=0}^{n}(-1)^{i}Q_{i}(\x;t)h_{n-i}(\x)\right)u^{n}
&=\left(\sum_{n=0}^{\infty}(-1)^{n}Q_{n}(\x;t)u^{n}\right) \left(\sum_{n=0}^{\infty}h_{n}(\x)u^{n}\right).
\end{align}
Similarly, the right-hand side of \eqref{lem-equi-eq2} becomes
\begin{align}\label{eq:gener-equi-fian-1-2}
&\sum_{n=0}^{\infty}\left( \sum_{i+j+m=n}(-1)^{m+j}e_{j}(\x)e_{m}(\x)P_{i}(\x;t) \right)u^{n}
=\left(\sum_{n=0}^{\infty}(-1)^{n}e_{n}(\x)u^{n}\right)^2\left(\sum_{n=0}^{\infty}P_{n}(\x;t)u^{n}\right).
\end{align}
From Stanley \cite[Section 7] {Stanley1999}, it is well known that
\begin{align*}
\sum_{n=0}^{\infty}h_{n}(\x)u^{n}=\prod_{i=1}^{\infty}\frac{1}{1-x_{i}u}
\qquad \text{and}
\qquad 
\sum_{n=0}^{\infty}e_{n}(\x)u^{n}=\prod_{i=1}^{\infty}(1+x_{i}u).
\end{align*}
This implies 
\begin{align}\label{eq:gener-equi-fian-1-3}
\left(\sum_{n=0}^{\infty}h_{n}(\x)u^{n}\right)\left(\sum_{n=0}^{\infty}(-1)^{n}e_{n}(\x)u^{n}\right)=1.
\end{align}
Hence, combining \eqref{lem-equi-eq2}, \eqref{eq:gener-equi-fian-1-1}, \eqref{eq:gener-equi-fian-1-2}, and \eqref{eq:gener-equi-fian-1-3}, we deduce
\begin{align}\label{eq:gener-equi-fian-1-4}
\sum_{n=0}^{\infty}(-1)^{n}Q_{n}(\x;t)u^{n}
=\left(\sum_{n=0}^{\infty}(-1)^{n}e_{n}(\x)u^{n}\right)^3\left(\sum_{n=0}^{\infty}P_{n}(\x;t)u^{n}\right).
\end{align}
Recalling that $e_{n}(\x)=e_{n}[X]$, and then applying Lemma \ref{lem:hung-1}, we get that
\begin{align*}
\left(\sum_{n=0}^{\infty}(-1)^{n}e_{n}(\x)u^{n}\right)^3
=&\sum_{n=0}^{\infty}(-1)^{n}\left(\sum_{i+j+k=n}e_{i}[X]e_{j}[X]e_{k}[X]\right)u^{n}\\
=&\sum_{n=0}^{\infty}(-1)^{n}e_{n}[3X]u^{n}.
\end{align*}
Thus, rewriting \eqref{eq:gener-equi-fian-1-4} gives
\begin{align}
\sum_{n=0}^{\infty}(-1)^{n}Q_{n}(\x;t)u^{n}
=&\left(\sum_{n=0}^{\infty}(-1)^{n}e_{n}[3X]u^{n}\right)
\left(\sum_{n=0}^{\infty}P_{n}(\x;t)u^{n}\right) \notag \\
&=\sum_{n=0}^{\infty}
\sum_{i=0}^{n}(-1)^{n-i}e_{n-i}[3X]P_{i}(\x;t)u^n.\label{lem-equi-eq3}
\end{align}
By comparing coefficients of $u^{n}$ on both sides of equation \eqref{lem-equi-eq3}, we complete the proof.
\end{proof}

We now present an alternative expression for $Q^{\mathfrak{S}_n}_{T_n}(t)$, in close analogy with the formula of Xie and Zhang for $P^{\mathfrak{S}_n}_{T_n}(t)$~\cite{xie2019equivariant}. Although distinct from the expression in Theorem~\ref{thm-equi}, this expression will be a key ingredient in the proofs of Theorems~\ref{thm-equi} and \ref{thm-ord}.

\begin{proposition}\label{equi-pro}
For any nonnegative integer $n$, we have
\begin{align*}
 {Q^{\mathfrak{S}_n}_{T_n}}(t)
 = \sum_{k=0}^{\lfloor \frac{n}{2} \rfloor }\left(\sum_{i=0}^{n-2k} \Ind^{\mathfrak{S}_n}_{\mathfrak{S}_{i} \times \mathfrak{S}_{n-i}} V_{(1^i)} \otimes V_{(2^{k}, 1^{n-2k-i})} \right) t^{k}.
\end{align*}
\end{proposition}

\begin{proof}
Applying the Frobenius characteristic map, it suffices to show that for all $n \geq 0$, 
\begin{align}\label{equi-pro-eq1}
 Q_n(\x;t)= \sum_{k=0}^{\lfloor \frac{n}{2} \rfloor }\left(\sum_{i=0}^{n-2k} e_{i}(\x) s_{(2^{k}, 1^{n-2k-i})}(\x) \right) t^{k}.
\end{align}
From Xie and Zhang \cite[Theorem 2.1]{xie2019equivariant}, we have

\[
P_i(\x;t)=
\begin{dcases}
1, & i=0,\\[2pt]
h_i(\x)\;+\;t \displaystyle\sum_{j=2}^{i} h_{i-j}(\x)\!
\displaystyle\sum_{k=0}^{\left\lfloor \frac{j}{2} \right\rfloor - 1}
s_{\,j-2k,\,2^{k}}(\x)\, t^{k}, & i \ge 1.
\end{dcases}
\]

Substituting this expression for $P_i(\x;t)$ into \eqref{lem-equi-equ}, as given in Lemma \ref{lem-equi}, yields
\begin{align*} 
Q_n(\x;t)&=F_{1}(\x)+F_{2}(\x;t),
\end{align*}
where
\begin{align*} 
F_{1}(\x):=\sum_{i=0}^{n}(-1)^{i}e_{n-i}[3X]h_{i}(\x)
\end{align*}
and 
\begin{align} \label{eq:thm-equi-thm-equi-eq2-2}
F_{2}(\x;t):=\sum_{i=2}^{n}(-1)^{i}e_{n-i}[3X]\cdot t\sum_{j=2}^{i}h_{i-j}(\x)\sum_{k=0}^{\lfloor\frac{j}{2}\rfloor-1}s_{j-2k,2^k}(\x)t^{k}.
\end{align}
We now simplify $F_1(\x)$ and $F_2(\x;t)$.
By Lemma~\ref{lem:hung-1}, we obtain
\begin{align}\label{eq:f_1}
F_{1}(\x)=\sum_{i=0}^{n}(-1)^i e_{n-i}[3X]h_i[X]=e_n[2X].
\end{align}
For $F_2(\x;t)$, interchanging the summation order over $i, j$, and $k$ in \eqref{eq:thm-equi-thm-equi-eq2-2} gives
\begin{align}\label{thm-equi-eq3}
F_{2}(\x;t)
=\sum_{k=0}^{\lfloor\frac{n}{2}\rfloor-1}\left(\sum_{j=2k+2}^{n}s_{j-2k,2^k}(\x)\sum_{i=j}^{n}(-1)^{i}e_{n-i}[3X]h_{i-j}(\x)
\right)t^{k+1}.
\end{align}
Let
$$
g(\x): = \sum_{i=j}^{n} (-1)^i e_{n-i}[3X]h_{i-j}(\x) ,
$$
which represents the innermost sum in \eqref{thm-equi-eq3}. We have
\begin{align*}
g(\x)
=\sum_{i=0}^{n-j}(-1)^{i+j}e_{n-i-j}[3X]h_{i}(\x) 
=(-1)^j\sum_{i=0}^{n-j}(-1)^i e_{n-i-j}[3X]h_{i}[X].
\end{align*}
By Lemma \ref{lem:hung-1}, we get
\begin{align*}
g(\x)=(-1)^j e_{n-j}[2X].
\end{align*}
Substituting this expression for $g(\x)$ back into \eqref{thm-equi-eq3}, we obtain 
\begin{align}\label{eq:f_2}
F_{2}(\x;t)&= \sum_{k=0}^{\lfloor\frac{n}{2}\rfloor-1}\left(\sum_{j=2k+2}^{n}(-1)^{j}s_{j-2k,2^k}(\x)e_{n-j}[2X]\right)t^{k+1}\nonumber\\
&=\sum_{k=1}^{\lfloor\frac{n}{2}\rfloor}\left(\sum_{j=0}^{n-2k}(-1)^{j}s_{(j+2,2^{k-1})}(\x)e_{n-2k-j}[2X]\right)t^{k}. 
\end{align}
Substituting the simplified expressions \eqref{eq:f_1} and \eqref{eq:f_2} for $F_{1}(\x)$ and $F_{2}(\x;t)$ into $Q_n(\x;t)$ yields 
\begin{align}\label{equi-pro-eq2}
Q_n(\x;t)= e_{n}[2X]+\sum_{k=1}^{\lfloor\frac{n}{2}\rfloor}\left(\sum_{j=0}^{n-2k}(-1)^{j}s_{(j+2,2^{k-1})}(\x)e_{n-2k-j}[2X]\right)t^{k}. 
\end{align}

It remains to show that
\eqref{equi-pro-eq1} and \eqref{equi-pro-eq2} are equivalent.
We verify this by comparing the coefficients of $t^k$
for each $0 \leq k \leq \lfloor n/2 \rfloor$.
For the constant term, recall that $s_{(1^n)}(\x) = e_n(\x)$.
Using Lemma~\ref{lem:hung-1}, the coefficient of $t^0$ in
\eqref{equi-pro-eq1} is
 $$\sum_{i=0}^{n} e_{i}(\x) s_{(1^{n-i})}(\x) = e_n[2X],$$
which agrees with the constant term in \eqref{equi-pro-eq2}.
For $1 \leq k \leq \lfloor \frac{n}{2} \rfloor $, 
it suffices to show that
\begin{align}\label{eq-sch-1}
 \sum_{i=0}^{n-2k} e_{i}(\x) s_{(2^{k}, 1^{n-2k-i})}(\x) 
 =\sum_{j=0}^{n-2k}(-1)^{j}s_{(j+2,2^{k-1})}(\x)e_{n-2k-j}[2X].
\end{align}
Applying Lemma~\ref{lem:hung-1} again, the right-hand side of \eqref{eq-sch-1} can be rewritten as
\begin{align*}
 &\sum_{j=0}^{n-2k} (-1)^{j} s_{(j+2,2^{k-1})}(\x) \sum_{i=0}^{n-2k-j} e_{i}(\x) e_{n-2k-j-i}(\x)
 \\
 &\qquad \qquad \qquad \qquad \qquad=\sum_{i=0}^{n-2k} e_{i}(\x) \sum_{j=0}^{n-2k-i}(-1)^{j} e_{n-2k-i-j}(\x) s_{(j+2,2^{k-1})}(\x).
\end{align*}
Hence, it remains to verify that for all $0 \leq i \leq n-2k$,
$$\sum_{j=0}^{n-2k-i}(-1)^{j} s_{(1^{n-2k-i-j})}(\x) s_{(j+2,2^{k-1})}(\x) = s_{(2^{k}, 1^{n-2k-i})}(\x).$$
Setting $m = n-2k-i$ in Lemma \ref{lemma-pieri-1} completes the proof.
\end{proof}

We are now ready to prove Theorem~\ref{thm-equi}.

\begin{proof}[Proof of Theorem \ref{thm-equi}]
By the Frobenius characteristic map,
it suffices to show that for $n \geq 0$,
\begin{align*}
 Q_n(\x;t)= \sum_{k=0}^{\lfloor \frac{n}{2} \rfloor }\left(\sum_{i=0}^{k} \sum_{j=k-i}^{\lfloor \frac{n-3i}{2} \rfloor} (n-3i-2j+1) s_{(3^i,2^j,1^{n-3i-2j})}(\x) \right) t^{k}.
\end{align*}
From \eqref{equi-pro-eq1}, this reduces to verifying that, for $0 \le k \le \lfloor n/2 \rfloor$,
$$\sum_{i=0}^{n-2k} e_{i}(\x) s_{(2^{k}, 1^{n-2k-i})}(\x) = \sum_{i=0}^{k} \sum_{j=k-i}^{\lfloor \frac{n-3i}{2} \rfloor} (n-3i-2j+1) s_{(3^i,2^j,1^{n-3i-2j})}(\x). $$
The claim then follows directly from Lemma~\ref{lem-pire-2}.
\end{proof}

\subsection{Proof of Theorem \ref{thm-ord}}

The aim of this subsection is to prove Theorem~\ref{thm-ord} 
by combining Proposition~\ref{equi-pro} with the classical hook-length formula. 

We begin by recalling some basic definitions and results. 
Let $\lambda=(\lambda_1,\lambda_2,\ldots,\lambda_{\ell(\lambda)})$ be a partition, where $\ell(\lambda)$ denotes the number of its nonzero parts.
The partition $\lambda$ can be represented by a left-justified array of cells, where the $i$-th row contains $\lambda_i$ cells. 
This array is called the Ferrers (or Young) diagram of $\lambda$.
Each cell in row $i$ and column $j$ is denoted by $(i,j)$, with rows numbered from top to bottom and columns from left to right. 
The hook-length of a cell $(i,j)$, denoted by $h(i,j)$, is the total number of cells directly to the right or directly 
below $(i,j)$, including the cell itself. 
For example, the hook-length of the cell $(1,2)$ 
in the Young diagram of $(4,4,2)$ is $5$, 
as illustrated in Figure~\ref{Young}.

\begin{figure}[h]
\centering
\begin{tikzpicture}[scale=0.6]
\foreach \x in {1,2,3,4} {
 \ifnum\x>1 
 \filldraw[fill=gray] (\x,0) rectangle (\x+1,1);
 \else
 \draw (\x,0) rectangle (\x+1,1);
 \fi
}

\foreach \x in {1,2,3,4} {
 \ifnum\x=2 
 \filldraw[fill=gray] (\x,-1) rectangle (\x+1,0);
 \else
 \draw (\x,-1) rectangle (\x+1,0);
 \fi
}

\foreach \x in {1,2} {
 \ifnum\x=2 
 \filldraw[fill=gray] (\x,-2) rectangle (\x+1,-1);
 \else
 \draw (\x,-2) rectangle (\x+1,-1);
 \fi
}
\end{tikzpicture}
\captionof{figure}{The hook set of $(1,2)$ composed of gray cells.}
\label{Young}

\end{figure}
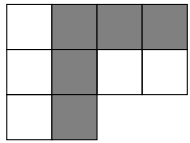

The classical hook-length formula, discovered by Frame, Robinson, 
and Thrall~\cite{frame1954CJM}, is stated below.

\begin{lemma}[\cite{frame1954CJM}]\label{lem:lem-sagan}
For any partition $\lambda \vdash n$, we have
\begin{align*}
f^{\lambda} = \frac{n!}{\displaystyle\prod_{(i,j)\in \lambda} h{(i,j)}},
\end{align*}
where $f^\lambda$ denotes the number 
of standard Young tableaux of shape $\lambda$. %
\end{lemma}

We are now in a position to prove Theorem~\ref{thm-ord}.

\begin{proof}[Proof of Theorem \ref{thm-ord}.]
Taking the graded dimension of ${Q^{\mathfrak{S}_n}_{T_n}}(t)$ in Proposition~\ref{equi-pro}, we obtain
\begin{align}\label{thm-ord-1}
Q_{T_n}(t)
&=\sum_{k=0}^{\lfloor \frac{n}{2} \rfloor }
\sum_{i=0}^{n-2k} 
\dim \mathrm{Ind}_{\mathfrak{S}_{i} \times \mathfrak{S}_{n-i}}^{\mathfrak{S}_{n}}\left(V_{(1^i)} \otimes V_{(2^k,1^{n-2k-i})}\right) 
 t^{k}\nonumber\\
&=\sum_{k=0}^{\lfloor \frac{n}{2} \rfloor }
\sum_{i=0}^{n-2k} 
\lvert \mathfrak{S}_n : \mathfrak{S}_i \times \mathfrak{S}_{n-i} \rvert \times \dim V_{(1^i)} \times \dim V_{(2^k,1^{n-2k-i})} 
 t^{k},
\end{align}
where $\lvert \mathfrak{S}_n : \mathfrak{S}_i \times \mathfrak{S}_{n-i} \rvert$ denotes the index of $\mathfrak{S}_i \times \mathfrak{S}_{n-i}$ in $\mathfrak{S}_n$, 
which equals $\binom{n}{i}$. 
Recall that for any partition $\lambda$, the dimension of the irreducible 
representation $V_\lambda$ of $\mathfrak{S}_n$ coincides with $f^\lambda$. 
Applying Lemma~\ref{lem:lem-sagan} gives
\begin{align}\label{eq:hook-length-special}
\dim V_{(1^i)}=1
\qquad \text{and} 
\qquad 
\dim V_{(2^k,1^{n-2k-i})} = \frac{(n-2k-i+1)(n-i)!}{k!(n-k-i+1)!}.
\end{align}
Substituting \eqref{eq:hook-length-special} into \eqref{thm-ord-1} yields
\begin{align*}
 Q_{T_n}(t)&=
 \sum_{k=0}^{\lfloor \frac{n}{2} \rfloor }\sum_{i=0}^{n-2k} \binom{n}{i} \frac{(n-2k-i+1)(n-i)!}{k!(n-k-i+1)!} t^{k}\\
&=\sum_{k=0}^{\lfloor \frac{n}{2} \rfloor }\sum_{i=0}^{n-2k} \frac{(n-2k-i+1)}{n+1} \dbinom{n+1}{k,i,n-k-i+1} t^{k}.
\end{align*}
Applying the index substitution $i \mapsto i+2k$ yields 
the desired expression~\eqref{QTN}. Thus we complete the proof.
\end{proof}

Similarly, combining Theorem~\ref{thm-equi} with Lemma~\ref{lem:lem-sagan}, 
we obtain the following result, which gives an alternative expression for 
$Q_{T_n}(t)$, different from that in Theorem~\ref{thm-ord}.

\begin{theorem}
For any thagomizer matroid $T_{n}$ with $n \geq 0$, we have
\begin{align*}
 Q_{T_{n}}(t)= \sum_{k=0}^{\lfloor \frac{n}{2} \rfloor }\left(\sum_{i=0}^{k} \sum_{j=k-i}^{\lfloor \frac{n-3i}{2} \rfloor} \frac{(j+1)(n-3i-j+2)(n-3i-2j+1)^2 n!}{i!(i+j+1)!(n-2i-j+2)!} \right) t^{k}.
\end{align*}
\end{theorem}

\begin{proof}
For any nonnegative integer $n$, taking the graded dimension of ${Q^{\mathfrak{S}_n}_{T_n}}(t)$ in Theorem~\ref{thm-equi}, we deduce that
\begin{align}\label{thm-ordnew-1}
 {Q_{T_n}}(t) = \sum_{k=0}^{\lfloor \frac{n}{2} \rfloor }\sum_{i=0}^{k} \sum_{j=k-i}^{\lfloor \frac{n-3i}{2} \rfloor} (n-3i-2j+1) \dim V_{(3^i,2^j,1^{n-3i-2j})} t^{k}.
\end{align}
Using the relation $\dim V_{\lambda} = f^\lambda$ together with 
Lemma~\ref{lem:lem-sagan}, we obtain
\begin{align}\label{eq:dim-special-partition}
\dim V_{(3^i,2^j,1^{n-3i-2j})}=\frac{(j+1)(n-3i-j+2)(n-3i-2j+1) n!}{i!(i+j+1)!(n-2i-j+2)!}.
\end{align}
Substituting \eqref{eq:dim-special-partition} into \eqref{thm-ordnew-1} leads to
\begin{align*}
 {Q_{T_n}}(t) &= \sum_{k=0}^{\lfloor \frac{n}{2} \rfloor }\sum_{i=0}^{k} \sum_{j=k-i}^{\lfloor \frac{n-3i}{2} \rfloor} (n-3i-2j+1) \frac{(j+1)(n-3i-j+2)(n-3i-2j+1) n!}{i!(i+j+1)!(n-2i-j+2)!} t^{k}\\
 &=\sum_{k=0}^{\lfloor \frac{n}{2} \rfloor }\sum_{i=0}^{k} \sum_{j=k-i}^{\lfloor \frac{n-3i}{2} \rfloor} \frac{(j+1)(n-3i-j+2)(n-3i-2j+1)^2 n!}{i!(i+j+1)!(n-2i-j+2)!} t^{k},
\end{align*}
which completes the proof.
\end{proof}

\section{Generating functions}\label{sec-inv}

In this section, we present an alternative proof of Theorem~\ref{thm-ord}.
Our approach is based on a relationship between the inverse Kazhdan--Lusztig polynomials and the Kazhdan--Lusztig polynomials of thagomizer matroids, together with ordinary generating functions.
Let 
\begin{align*}
\Phi(t,u):=\sum_{n=0}^{\infty}P_{T_n}(t)u^{n+1}
\qquad \text{and} 
\qquad
\Psi(t,u):=\sum_{n=0}^{\infty}Q_{T_n}(t)u^{n}.
\end{align*}
An explicit expression for $\Phi(t,u)$ was given by Gedeon~\cite{gedeon2017kazhdan}. 
We now derive a corresponding expression for $\Psi(t,u)$.

\begin{proposition} \label{prop-gene-4.1}
We have the following closed form,
\begin{align}\label{genQTN}
\Psi(t,u)=\frac{1-3u-\sqrt{(1-4t)u^{2}-2u+1}}{2u\left(tu+2u-1\right)}. 
\end{align}
\end{proposition}

\begin{proof}
By \cite[Theorem 1.3]{gao2021inverse}, for every matroid $\M$, we have 
\begin{align}\label{eq:thg_gene_1}
 \sum_{ F \in \mathcal{L}(\M)} (-1)^{\rk( \M|_F)}P_{ \M|_F}(t) \, Q_{\M/{F}}(t)=0.
\end{align}
Applying \eqref{eq:thg_gene_1} to the thagomizer matroid $T_n$ yields 
\begin{align}\label{eq:thg_gene_21}
\sum_{i=0}^{n}\binom{n}{i}2^{i} (-1)^{\rk( B_{i})}P_{B_i}(t) \,Q_{T_{n-i}}(t)
+\sum_{i=0}^{n}\binom{n}{i}(-1)^{\rk (T_{i})}P_{T_i}(t)\, Q_{B_{n-i}}(t)=0. 
\end{align}
By \cite[Corollary 2.10]{elias2016kazhdan} and \cite[Corollary 3.2]{gao2021inverse}, we have
\begin{align}\label{boolean-kl-ikl}
P_{B_n}(t)=Q_{B_n}(t)=1 \qquad \text{for all} ~ n \geq 0.
\end{align}
Substituting \eqref{boolean-kl-ikl} into \eqref{eq:thg_gene_21} yields
\begin{align}\label{eq:thg_gene_3}
\sum_{i=0}^{n}\binom{n}{i}(-2)^{i}Q_{T_{n-i}}(t)
+\sum_{i=0}^{n}\binom{n}{i}(-1)^{i+1}P_{T_i}(t)=0.
\end{align} 
Equivalently, \eqref{eq:thg_gene_3} can be written as
\begin{align} \label{lem-ord-eq}
 \sum_{i=0}^{n}\binom{n}{i}(-2)^{n-i}Q_{T_{i}}(t)=\sum_{i=0}^{n}\binom{n}{i}(-1)^{i}P_{T_i}(t).
\end{align}

We now multiply both sides of \eqref{lem-ord-eq} by $u^{n+1}$ and sum over $n \geq 0$. Let $L(t,u)$ and $R(t,u)$ denote the left-hand and right-hand sides of the resulting identity, respectively. For the left-hand side, we have
\begin{align*}
L(t,u)=\sum_{n=0}^{\infty}\sum_{i=0}^{n}\binom{n}{i}(-2)^{n-i}Q_{T_i}(t)u^{n+1} 
=\sum_{i=0}^{\infty}Q_{T_i}(t)u^{i+1}\sum_{n=i}^{\infty}\binom{n}{i}(-2u)^{n-i}. 
\end{align*}
Setting $m=n-i$ gives 
\begin{align*}
L(t,u)=\sum_{i=0}^{\infty}Q_{T_i}(t)u^{i+1}\sum_{m=0}^{\infty}\binom{m+i}{i}(-2u)^{m}.
\end{align*}
By the standard identity
\begin{align}\label{geniden}
\sum_{\ell=0}^{\infty}\binom{\ell+r}{r}x^{\ell}=\frac{1}{(1-x)^{r+1}}
 \qquad \text{for all} ~ |x| < 1,
\end{align}
with $(\ell,r,x) = (m,i,-2u)$, we obtain
\begin{align*}
\sum_{m=0}^{\infty}\binom{m+i}{i}(-2u)^{m}
=\frac{1}{(1+2u)^{i+1}}.
\end{align*}
Thus
\begin{align}\label{L(t,u)}
L(t,u)=\sum_{i=0}^{\infty}Q_{T_i}(t)\left(\frac{u}{1+2u}\right)^{i+1}=\frac{u}{1+2u}\Psi(t,\frac{u}{1+2u}).
\end{align}
For the right-hand side, similarly,
\begin{align*}
R(t,u)=\sum_{n=0}^{\infty}\sum_{i=0}^{n}\binom{n}{i}(-1)^{i}P_{T_i}(t)u^{n+1}=\sum_{i=0}^{\infty}(-1)^{i}P_{T_i}(t)u^{i+1}\sum_{n=i}^{\infty}\binom{n}{i}u^{n-i}.
\end{align*}
Setting $m = n-i$ and using \eqref{geniden} again,
\begin{align*}
 R(t,u)=\sum_{i=0}^{\infty}(-1)^iP_{T_i}(t)u^{i+1}\sum_{m=0}^{\infty}\binom{m+i}{i}u^{m}
 =-\sum_{i=0}^{\infty}P_{T_i}(t)\left(\frac{u}{u-1}\right)^{i+1}.
\end{align*}
This simplifies to
\begin{align}\label{R(t,u)}
 R(t,u)=-\Phi(t,\frac{u}{u-1}).
\end{align}
From \eqref{L(t,u)} and \eqref{R(t,u)}, we have
\begin{align*}
 \Psi(t,\frac{u}{1+2u})=-\frac{1+2u}{u} \Phi(t,\frac{u}{u-1}),
\end{align*}
which can be rewritten as
\begin{align}\label{genrepq}
 \Psi(t,u)=-\frac{1}{u}\Phi(t,\frac{u}{3u-1}).
\end{align}
Recall that by \cite[Theorem 1]{gedeon2017kazhdan}, 
\begin{align}\label{genrep}
 \Phi(t,u)=\frac{1-\sqrt{1-4u(1-u+tu)}}{2(1-u+tu)}.
\end{align}
Substituting \eqref{genrep} into \eqref{genrepq} and considering sufficiently small $u$, we obtain \eqref{genQTN}. This completes the proof.
\end{proof}

From the explicit expression for $\Psi(t,u)$, we deduce the following recurrence for $Q_{T_n}(t)$.

\begin{proposition}\label{prop-recu-th-inv}
The sequence $\{Q_{T_n}(t)\}_{n=0}^{\infty}$ satisfies the recurrence relation 
\begin{align*}
 &(n+1)(t+2)(4t-1)Q_{T_{n}}(t)-(2nt-5n-t-11)Q_{T_{n+1}}(t)\\
 &\qquad \qquad \qquad \qquad -(nt+4n+4t+13)Q_{T_{n+2}}(t)+(n+4)Q_{T_{n+3}}(t)=0 \text{~~~for~all~} n\geq 0,
\end{align*}
with initial conditions $Q_{T_{0}}(t)=1$, $Q_{T_{1}}(t)=2$, and $Q_{T_{2}}(t)=t+4$.
\end{proposition}

\begin{proof}
By Proposition~\ref{prop-gene-4.1}, the generating function $\Psi(t,u)$ of the sequence $\{Q_{T_n}(t)\}$ has the closed form \eqref{genQTN}. 
The initial conditions follow from the expansion at $u=0$ (for instance, in Mathematica ):
\begin{mma}
\In \Psi[t\_,u\_]:=\frac{1-3u-\sqrt{(1-4t)u^2-2u+1}}{2u(tu+2u-1)};\\
\end{mma}
\begin{mma}
\In |CoefficientList|[|Series|[\Psi[t,u],\{u,0,2\}],u]\\
\end{mma}
\begin{mma}
\Out \{1,2,4+t\}\\
\end{mma}
Furthermore, one checks that $\Psi(t,u)$ satisfies the following first-order linear differential equation:

\begin{align}\label{eq:pde-1}
 & \big((4t^2+7t-2)u^3+(3t+6)u^2-(2t+5)u+1\big)\Psi(t,u)+(2t+1)u-1\nonumber
 \\
 & \qquad \qquad \qquad \qquad +u\big((4t^2+7t-2)u^3-(2t-5)u^2-(t+4)u+1\big) \Psi_{u}(t,u) =0.
\end{align}
Substituting the power-series expansion $\Psi(t,u)=\sum_{n=0}^{\infty} Q_{T_n}(t)\,u^{n}$ into \eqref{eq:pde-1} yields the following functional equation for the sequence $\{Q_{T_n}(t)\}_{n\ge 0}$:
\begin{align}\label{Q[n]rec-1}
 & \big((4t^2+7t-2)u^3+(3t+6)u^2-(2t+5)u+1\big)\sum_{n=0}^{\infty}Q_{T_n}(t)u^{n}+(2t+1)u-1\nonumber\\
& \qquad \qquad \qquad +\big((4t^2+7t-2)u^3-(2t-5)u^2-(t+4)u+1\big) \sum_{n=1}^{\infty}nQ_{T_n}(t)u^{n} =0. 
\end{align}
For $n \ge 0$, extracting the coefficient of $u^{n+3}$ from \eqref{Q[n]rec-1} yields precisely the recurrence relation stated above.
Hence $\{Q_{T_n}(t)\}$ satisfies the stated recurrence with the given initial conditions, which completes the proof.

\end{proof}

We now provide an alternative proof of Theorem~\ref{thm-ord}.

\begin{proof}[Alternative proof of Theorem \ref{thm-ord}.]
Fix $n\geq 0$. Define
$f_n(t):=\sum_{k=0}^{\infty}a_{n,k}t^k,$
where 
\begin{align*} 
a_{n,k}=
\sum_{i=2k}^{n}\frac{n-i+1}{n+1} \dbinom{n+1}{k,i-2k,n+k-i+1}
\end{align*}
for $0\leq k \leq \lfloor \frac{n}{2} \rfloor$, and set $a_{n,k}=0$ otherwise. For completeness, we also define $a_{n,k}=0$ whenever 
$k<0$ or $n<2k$. 
Let $c_{n,k}$ denote the coefficient of $t^k$ in $Q_{T_n}(t)=\sum_{k=0}^{\infty}c_{n,k}t^k$, with the same convention $c_{n,k}=0$ for $k<0$ or $n<2k$. 
To prove $Q_{T_n}(t)=f_n(t)$, it suffices to show that the arrays $\{c_{n,k}\}$ and $\{a_{n,k}\}$ satisfy the same recurrence and the same initial conditions.
Direct computation for 
$n=0,1,2$ confirms that $f_n(t)=Q_{T_n}(t)$ in these cases.

From Proposition~\ref{prop-recu-th-inv}, comparing coefficients of 
$t^k$ yields the recurrence
\begin{align}\label{c[n,k]rec}
 (n+4)c_{n+3,k}
 =&
 (4n+13)c_{n+2,k}
 +(n+4)c_{n+2,k-1}-(5n+11)c_{n+1,k}+(2n-1)c_{n+1,k-1}\nonumber\\
&\quad+2(n+1)c_{n,k}-7(n+1)c_{n,k-1}-4(n+1)c_{n,k-2}.
\end{align}
It remains to verify that $a_{n,k}$ satisfies the same relation. For this purpose, we employ the Mathematica package 
\texttt{HolonomicFunctions} \cite{koutschan2009advanced}. The command 
\texttt{Annihilator[}\textit{expr}\texttt{]} computes annihilating 
operators for a given summand \textit{expr}. Applying it to $a_{n,k}$ gives

\begin{mma}
\In <<|RISC| ~\grave{} |HolonomicFunctions|~\grave{};\\
\end{mma}

\begin{mma}
\In |ann|=|Annihilator|[|Sum|[\frac{n-i+1}{n+1}|Multinomial|[k,i-2k,n+k-i+1], \{i,2k,n\}],\linebreak
\{S[k],S[n]\}] \\
\end{mma}

\begin{mma}
\Out \{(n+1) \left(11 k^2-k (13 n+18)+4 (n+1) (n+2)\right)-2 k (k+1) (n+1)
 S_k+(k-n-2) (2 k-n-1) (3 k-2 (n+1)) S_n,(2 k-n-2) (k-n-3) S_n^2+(n+2) (5
 k-3 n-7) S_n+2 (n+1) (n+2)\}\\
\end{mma}
\noindent Here $S_n$ and $S_k$ denote the forward shift operators in $n$ and $k$, respectively.
We rewrite the recurrence~\eqref{c[n,k]rec} by shifting 
$k \mapsto k-2$, obtaining
\begin{align*}
 (n+4)c_{n+3,k+2}-(4n+13)c_{n+2,k+2}
 -(n+4)c_{n+2,k+1}+(5n+11)c_{n+1,k+2}
 -(2n-1)c_{n+1,k+1}\\
 -2(n+1)c_{n,k+2}+7(n+1)c_{n,k+1}+4(n+1)c_{n,k}=0.
\end{align*}
In operator form, this corresponds to the Ore polynomial
\begin{align*}
 & (n+4) S[k]^2 S[n]^3-(4 n+13) S[k]^2 S[n]^2-(n+4) S[k] S[n]^2+(5 n+11) S[k]^2
 S[n]\\
 & \qquad \qquad\qquad
 -(2 n-1) S[k] S[n]-2 (n+1) S[k]^2+7 (n+1) S[k]+4 (n+1).
\end{align*}
We verify that this operator reduces to $0$ modulo the 
annihilating ideal \texttt{ann},
which is accomplished by 

{\renewcommand\baselinestretch{1.3}
\begin{mma}
\In |OreReduce|[(n+4) S[k]^2 S[n]^3-(4 n+13) S[k]^2 S[n]^2-(n+4) S[k] S[n]^2+(5 n+11) S[k]^2 S[n]-(2 n-1) S[k] S[n]-2 (n+1) S[k]^2+7 (n+1) S[k]+4 (n+1),|ann|]\\
\end{mma}
}

\begin{mma}
\Out 0 \\
\end{mma}

This confirms that $a_{n,k}$ satisfies the recurrence~\eqref{c[n,k]rec}, 
and the proof is complete.
\end{proof}

Having established an explicit expression for $Q_{T_n}(t)$, we now apply the recent deletion formula for inverse Kazhdan--Lusztig polynomials due to Braden, Ferroni, Matherne, and Nepal~\cite{braden2025deletion}.
Combining this formula with our expression for $Q_{T_n}(t)$ yields the following corollary.

\begin{corollary}
Let $Q_{K_{2,n}}(t)$ denote the inverse
Kazhdan--Lusztig polynomial of the graphic
matroid of the complete bipartite graph $K_{2,n}$
for any $n\ge1$. Then
\begin{align}\label{QK2n}
 Q_{K_{2,n}}(t)=
\sum_{k=0}^{\lfloor\frac{n}{2}\rfloor} \sum_{i=2k}^{n} \frac{n-i+1}{n+1} \dbinom{n+1}{k,i-2k,n+k-i+1} t^{k}+(n-1)t-1.
\end{align}
\end{corollary}

\begin{proof}
By \cite[Theorem~1.4]{braden2025deletion}, for any element $i\in E$ that is not a coloop, one has
\begin{align}\label{dele-formula}
 Q_{\M}(t) = Q_{\M\smallsetminus i}(t) + (1+t) Q_{\M/i}(t) - \sum_{F\in \mathcal{T}_i} \tau(\M|_F/i)\, t^{\rk(F)/2}\, Q_{\M/F}(t),
\end{align}
where the set of flats is 
\begin{align*}
 \mathcal{T}_i = \mathcal{T}_i(\M) := \left\{ F \in \mathcal{L}(\M): i\in F \text{ and } F\smallsetminus\{i\}\notin \mathcal{L}(\M)\right\}.
\end{align*}
and
\begin{align*}
 \tau(\M) := 
 \begin{cases}
 [t^{(\rk(\M)-1)/2}]\, P_{\M}(t), & \text{if $\rk(\M)$ is odd},\\[6pt]
 0, & \text{if $\rk(\M)$ is even}.
 \end{cases}
\end{align*}
Applying \eqref{dele-formula} to the thagomizer matroid $T_n$ with the edge $AB$ chosen for deletion, we observe that
\begin{align*}
T_n \setminus AB = K_{2,n} \quad \text{and} \quad T_n / AB = B_n.
\end{align*}
The set $\mathcal{T}_i$ in this case consists of the $n$ distinct $3$-cycle flats.
Moreover, for each such flat $T_1$, we have
\begin{align*}
T_n|_{T_1}/AB \cong B_1 \quad \text{and} \quad T_n / T_1 \cong B_{n-1}.
\end{align*}
From \eqref{boolean-kl-ikl},
we obtain
\begin{align}\label{connec-tn-k2n}
Q_{K_{2,n}}(t)= Q_{T_n}(t)+(n-1)t-1.
\end{align}
Combining \eqref{connec-tn-k2n} with Theorem~\ref{thm-ord} completes the proof.
\end{proof}

\section{Log-concavity}\label{sec-log}

The objective of this section is to prove Theorem \ref{thm-log}, which establishes the log-concavity of inverse Kazhdan--Lusztig polynomials $Q_{T_{n}}(t)$ associated with thagomizer matroids.
Our proof follows the approach of Wu and Zhang~\cite{wu2023log}, who established the log-concavity of the Kazhdan--Lusztig polynomials $P_{T_n}(t)$ 
for all thagomizer matroids $T_n$, 
and is based on the following key idea.

Recall that
\begin{align}\label{cn,k}
c_{n,k}&=\sum_{i=2k}^{n} \frac{n-i+1}{n+1} \dbinom{n+1}{k,i-2k,n+k-i+1}.
\end{align}
Define
\begin{align} \label{eq:d_n,k}
d_{n,k}:=\sum_{i=2k}^{n}(n-i+1)\binom{n-k+1}{i-2k} 
\qquad 
\text{and} 
\qquad
b_{n,k}:=\frac{1}{n+1}\binom{n+1}{k},
\end{align}
so that $c_{n,k}=b_{n,k}d_{n,k}.$
To establish the log-concavity of the sequence $\{c_{n,k}\}_{k}$, it suffices to prove that both $\{d_{n,k}\}_k$ and $\{b_{n,k}\}_k$ are log-concave. The log-concavity of $\{b_{n,k}\}_k$ follows directly from standard properties of binomial coefficients. 
The main technical step is to prove the log-concavity of $\{d_{n,k}\}_k$, which is established in the following lemma.

\begin{lemma}\label{lem-dnk}
For any positive integers $n$ and $k$ with 
$n\geq 2k+2$, 
we have $$d^{2}_{n,k} \geq d_{n,k+1}d_{n,k-1}.$$
\end{lemma}

Our proof of Lemma~\ref{lem-dnk} is inspired by the computer-assisted proof of Moll's log-concavity conjecture on the coefficient sequences of Jacobi polynomials due to Kauers and Paule~\cite{kauers2007computer}.

\subsection{Recurrence relations of \texorpdfstring{$d_{n,k}$}{d\_\{n,k\}}}

In this subsection, we present several recurrence relations for $d_{n,k}$, 
which serve as the key to proving Lemma~\ref{lem-dnk}. 
These relations were derived using the package \texttt{HolonomicFunctions} by Koutschan \cite{koutschan2009advanced} for Mathematica.

\begin{lemma}
For any positive integers $n$ and $k$ with $n \geq 2k+1$, we have
\begin{align}
d_{n+1,k}&=-\frac{2(n-k+1)d_{n-1,k}-(3n-5k+4)d_{n,k}}{n-2k+1} \label{lem-rec1},\\ 
d_{n-1,k-1}&=-\frac{2(n-k+1)d_{n-1,k}-(2n-3k+3)d_{n,k}}{n-2k+1}\label{lem-rec3},\\
d_{n-1,k+1}&=\frac{\big(4n^2+(4-13k)n+11k^2-5k\big)d_{n-1,k}-\big(2n^2-7nk+6k^2\big)d_{n,k}}{2k(n-k)}\label{lem-rec2}.
\end{align}
\end{lemma}

\begin{proof}
The proof is analogous to the verification that $a_{n,k}$ satisfies 
the recurrence relation \eqref{c[n,k]rec}. 
We apply Koutschan's package \texttt{HolonomicFunctions} \cite{koutschan2009advanced} 
to the defining summation of $d_{n,k}$ and compute its annihilating ideal
by the following command.

\begin{mma}
\In |dnn|=|Annihilator|[|Sum|[(n-i+1)|Binomial|[n-k+1,i-2k], \{i,2k,n\}],\linebreak
\{S[k],S[n]\}] \\
\end{mma}

\begin{mma}
\Out \{2k(-1+k-n)S_{k} +\left(-6k^{2}+7k(1+n)-2(1+n)^{2}\right)S_{n}+ 11k^{2}+4(1+n)(2+n)-k(18+13n), 
(-2+2k-n)S_{n}^{2} +(7-5k+3n)S_{n} +2(-2+k-n)\}\\
\end{mma}

We verify that each of the claimed relations 
\eqref{lem-rec1}, \eqref{lem-rec3}, and \eqref{lem-rec2} reduces to zero modulo this annihilating ideal. 
This is done by the \texttt{OreReduce} command as follows.

\begin{mma}
\In |OreReduce|[(n-2k+2)S[n]^{2}-(3n-5k+7)S[n]+2(n-k+2),|dnn|]\\
\end{mma}

\begin{mma}
\Out 0 \\
\end{mma}

\begin{mma}
\In |OreReduce|[(2n-3k+2)S[k]S[n]-2(n-k+1)S[k]-(n-2k),|dnn|]\\
\end{mma}

\begin{mma}
\Out 0 \\
\end{mma}

\begin{mma}
\In |OreReduce|[(2n^{2}+(4-7k)n+6k^{2}-7k+2)S[n]+2k(n-k+1)S[k]-4n^{2}+(13k-12)n-11k^{2}+18k-8,|dnn|]\\
\end{mma}

\begin{mma}
\Out 0 \\
\end{mma}

Since all three Ore reductions yield zero, 
the desired recurrences \eqref{lem-rec1}--\eqref{lem-rec2} are established. 
This completes the proof.
\end{proof}

\subsection{Lower bound of \texorpdfstring{$\frac{d_{n,k}}{d_{n-1,k}}$}{d\_\{n,k\}/d\_\{n-1,k\}}}

The aim of this subsection is to establish a lower bound for the ratio $\frac{d_{n,k}}{d_{n-1,k}}$ for all positive integers $n$ and $k$ satisfying $n \geq 2k+1$. To this end, we introduce the following notation.
Let 
\begin{align}
X(n,k):=&\frac{1}{2}\left(3+\frac{k-1}{2n-3k+3}+\frac{k+3}{n-2k}-\frac{k+2}{2n-3k}\right) \notag \\
&+\frac{\sqrt{y(n,k)}}{2(2n-3k+3)(2n-3k)(n-2k)}, \label{eq:x_n,k}
\end{align}
where 
\begin{align*}
y(n,k):=&\big(4n^2+(4-12k)n+9k^2-5k\big)\big(4n^4+(28-36k)n^3+(60-149k+117k^2)n^2\\
&+(36-174k+276k^2-162k^3)n+81k^4-171k^3+135k^2-45k\big).
\end{align*}
It is worth noting that both $y(n,k) > 0$ and $X(n,k) > 0$ hold for all positive integers $n$ and $k$ satisfying $n \geq 2k + 1$. 
The positivity of $y(n,k)$ can be verified either through a direct but lengthy symbolic computation or with the aid of the Mathematica command \texttt{Resolve} using the following input.
{\renewcommand\baselinestretch{1.5}
\begin{mma}
\In y[n\_,k\_]:=(4n^2+(4-12k)n+9k^2-5k)(4n^4+(28-36k)n^3+(60-149k+117k^2)n^2+(36-174k+276k^2-162k^3)n+81k^4-171k^3+135k^2-45k); \\
\In y[n,k]>0;\\
\In |Resolve|[|ForAll|[\{n,k\},n \geq 1 \&\& k\geq 1 \&\&n \geq 2k+1,\%]]\\
\end{mma}
}
\begin{mma}
\Out |True|\\
\end{mma}

We now verify that $X(n,k) > 0$. Under the assumption $n \geq 2k + 1$, the denominator
$(2n-3k+3)(2n - 3k)(n - 2k)$ is strictly positive.
Since the numerator of the square root is $y(n,k) > 0$, the entire radical term is positive. Moreover, each of the terms
$\frac{k-1}{2n-3k+3}$, $\frac{k+3}{n-2k}$, and $3 - \frac{k+2}{2n-3k}$ is positive. 
It follows that $X(n,k)$ is the sum of positive terms, and hence $X(n,k) > 0$.

\begin{lemma}\label{lem-log1}
For any positive integers $n$ and $k$ with $n \geq 2k+1 $, we have
$$\frac{d_{n,k}}{d_{n-1,k}} \geq X(n,k).$$
\end{lemma}

\begin{proof}
Fix $k \geq 1$. 
We proceed by induction on $n \geq 2k+1$.
For the base case $n=2k+1$, we verify that
$$\frac{d_{2k+1,k}}{d_{2k,k}} \geq X(2k+1,k).$$
From the explicit expressions of $d_{n,k}$ in~\eqref{eq:d_n,k} and $X(n,k)$ in~\eqref{eq:x_n,k}, we compute
\begin{align*}
&\frac{d_{2k+1,k}}{d_{2k,k}}-X(2k+1,k)\\
=&\frac{k^3+9k^2+30k+32-\sqrt{(k+2)(k^2+7k+8)(k^3+9k^2+22k+64)}}{2(k+2)(k+5)}.
\end{align*}
Observe that
\begin{align*}
&(k^3+9k^2+30k+32)^{2}-(k+2)(k^2+7k+8)(k^3+9k^2+22k+64) \\
=& 16k(k+1)(k+2)(k+5) > 0,
\end{align*}
for all positive integers $k$, which establishes the base case.

Now suppose the inequality holds for some $n \geq 2k+1$, i.e.,
\begin{align*}
\frac{d_{n,k}}{d_{n-1,k}} \geq X(n,k).
\end{align*}
We aim to prove the inequality also holds for $n+1$. Dividing both sides of recurrence~\eqref{lem-rec1} by $d_{n,k}$, we obtain
\begin{align*}
\frac{d_{n+1,k}}{d_{n,k}}
=-\frac{2(n-k+1)}{(n-2k+1)} \cdot \frac{d_{n-1,k}}{d_{n,k}}
+\frac{3n-5k+4}{n-2k+1}. 
\end{align*}
Since the coefficient $-\frac{2(n - k + 1)}{n - 2k + 1} < 0$ and $X(n,k) > 0$, applying the inductive hypothesis yields
\begin{align*}
\frac{d_{n+1,k}}{d_{n,k}}
\geq 
-\frac{2(n-k+1)}{(n-2k+1)}\cdot \frac{1}{X(n,k)}+\frac{3n-5k+4}{n-2k+1}.
\end{align*}

To complete the inductive step, it suffices to verify that the right-hand side above is greater than or equal to $X(n+1,k)$.
This inequality can be verified automatically using the Mathematica \texttt{Resolve} command via the following input.
{\renewcommand\baselinestretch{2.5}
\begin{mma}
\In X[n\_,k\_]:=\frac{1}{2}\left(3+\frac{k-1}{2n-3k+3}+\frac{k+3}{n-2k}-\frac{k+2}{2n-3k}\right)
\linebreak
+\frac{\sqrt{y(n,k)}}{2(2n-3k+3)(2n-3k)(n-2k)};\\

\In -\frac{2(n-k+1)}{(n-2k+1)X[n,k]}+\frac{3n-5k+4}{n-2k+1} \geq X[n+1,k];\\

\In |Resolve|[|ForAll|[\{n,k\},n \geq 1 \&\& k\geq 1 \&\&n \geq 2k+1,\%]]\\
\end{mma}
}
\begin{mma}
\Out |True|\\
\end{mma} 
This completes the proof.
\end{proof}

\subsection{Proof of Theorem \ref{thm-log}}

In this subsection, we first prove Lemma~\ref{lem-dnk}, and then proceed to establish Theorem~\ref{thm-log}.

\begin{proof}[Proof of Lemma \ref{lem-dnk}.]

For notational simplicity and clarity, we shall prove
for any positive integers $n$ and $k$ with 
$n\geq 2k+3$,
\begin{align}\label{eq:lemma-4.1}
d^{2}_{n-1,k} - d_{n-1,k+1}d_{n-1,k-1} \geq 0.
\end{align}
Using the recurrence relations~\eqref{lem-rec3} and~\eqref{lem-rec2}, we obtain
\begin{align*}
&2k(n-k)(n-2k+1)\left(d^{2}_{n-1,k} - d_{n-1,k+1}d_{n-1,k-1}\right)
=d^{2}_{n-1,k}f_{n,k}\left(\frac{d_{n,k}}{d_{n-1,k}}\right),
\end{align*}
where
\begin{align*}
f_{n,k}(x):=&(2n-3k+3)(2n-3k)(n-2k)x^2\\
&+\left(45k^3 -(87n+60)k^{2}+(56n^2+75n+15)k-12n(n+1)^2 \right)x\\
&-2\left(9k^3-3(7n+5)k^{2}+(16n^2+21n+5)k-4n(n+1)^2 \right).
\end{align*}
It can be verified that the discriminant of $f_{n,k}(x)$ is $y(n,k)$. As shown earlier, $y(n,k) > 0$, so $f_{n,k}$ has two distinct real roots, which are
\begin{align*}
z_{1}(n,k):=&\frac{1}{2}\left(3+\frac{k-1}{2n-3k+3}+\frac{k+3}{n-2k}-\frac{k+2}{2n-3k}\right)\\
&-\frac{\sqrt{y(n,k)}}{2(2n-3k+3)(2n-3k)(n-2k)}
\end{align*}
and
\begin{align*}
z_{2}(n,k):=&\frac{1}{2}\left(3+\frac{k-1}{2n-3k+3}+\frac{k+3}{n-2k}-\frac{k+2}{2n-3k}\right)\\
&+\frac{\sqrt{y(n,k)}}{2(2n-3k+3)(2n-3k)(n-2k)}.
\end{align*}
Since $2(2n-3k+3)(2n-3k)(n-2k) > 0$,
we have $z_1(n,k) < z_2(n,k)$ and the leading coefficient of $f_{n,k}(x)$ is positive. 
Moreover, one observes that $z_2(n,k)$ coincides with the previously defined $X(n,k)$. The desired inequality \eqref{eq:lemma-4.1} then follows from Lemma~\ref{lem-log1} and the non-negativity of the factor $2k(n-k)(n-2k+1)$. This completes the proof.
\end{proof}

Now, we proceed to prove Theorem \ref{thm-log}.

\begin{proof}[Proof of Theorem \ref{thm-log}.]
By Lemma~\ref{lem-dnk}, we have $d^{2}_{n,k} \geq d_{n,k+1}d_{n,k-1}$. It is also well known that the sequence of binomial coefficients $\binom{n}{k}$ is log-concave, i.e., 
$$\dbinom{n}{k}^2 \geq \dbinom{n}{k+1} \dbinom{n}{k-1}.$$
This implies the analogous inequality for $b_{n,k}$, namely, $b^{2}_{n,k} \geq b_{n,k+1}b_{n,k-1}$. 
Since $c_{n,k} = b_{n,k} d_{n,k}$, combining the two inequalities and using the fact that the product of two log-concave sequences is again log-concave, we obtain
 $c^{2}_{n,k} \geq c_{n,k+1}c_{n,k-1}$
as desired.
\end{proof}

Based on Theorem~\ref{thm-log}, we obtain the following corollary.

\begin{corollary}
For any positive integer $n$, the inverse Kazhdan--Lusztig polynomial $Q_{K_{2,n}}(t)$ of the graphic
matroid of the complete bipartite graph $K_{2,n}$ is log-concave.
\end{corollary}

\begin{proof}
By combining \eqref{connec-tn-k2n} with Theorem~\ref{thm-log}, it suffices to prove that for any positive integer $n$,
\begin{align}\label{last-eq}
(c_{n,1}+n-1)^2 \geq (c_{n,0}-1)c_{n,2}
\qquad\text{and}\qquad
c_{n,2}^2 \geq (c_{n,1}+n-1)c_{n,3}.
\end{align}
The first inequality follows directly from the known relation $c^2_{n,1} \geq c_{n,0} c_{n,2}$.
For the second inequality, we let  $k=1,2,3$ in~\eqref{cn,k} and  obtain
\begin{align*}
 c_{n,1} &= 2^{n-1}(n-2)+1, \\
 c_{n,2} &= \frac{n\bigl(2^n (n-5)+4(n+1)\bigr)}{8},\\
 c_{n,3}&= \frac{n(n-1)\bigl(2^n (n-8)+4n(n-1)+16\bigr)}{48}.
\end{align*}
Substituting these expressions into the second inequality of \eqref{last-eq} and simplifying yields
\begin{align*}
 c_{n,2}^2 - (c_{n,1}+n-1)c_{n,3}
 =& \frac{3 n^2 \bigl(2^n (n-5)+4 (n+1)\bigr)^2}{192} \\
 &- \frac{4n(n-1)\bigl(2^{n-1} (n-2)+n\bigr)
 \bigl(2^n (n-8)+4n(n-1)+16\bigr)}{192}.
\end{align*}
Thus, it suffices to show that for all positive integers $n$,
\begin{align*}
 3 n \bigl(2^n (n-5)+4 (n+1)\bigr)^2
 - 4 (n-1) \bigl(2^{n-1} (n-2)+n\bigr)
 \bigl(2^n (n-8)+4n(n-1)+16\bigr) \ge 0.
\end{align*} 
This nonnegativity can be  verified rigorously in Mathematica by quantifier elimination via \texttt{Resolve}.

\begin{mma}
\In 3 n (2^n (n-5)+4 (n+1))^2-4 (n-1) (2^{n-1} (n-2)+n)(2^n (n-8)+4 (n-1) n +16) \geq 0; \\
\end{mma}

\begin{mma}
\In |Resolve|[|ForAll|[\{n\},n \geq 1 \&\& n \in |Integers| ,\%]]\\
\end{mma}

\begin{mma}
\Out |True|\\
\end{mma}

This completes the proof.
\end{proof}

\vspace{4mm}

\noindent{\bf Acknowledgements.} 
The first author is supported by the National Science Foundation of China (No.11801447) and the National Science Foundation for Post-doctoral Scientists of China (No.2020M683544).
The third author is supported by the National Science Foundation of China (No.11901431).

\end{document}